%% file: Isotropic_deformations_pdf.tex
\documentclass[12pt]{amsart}  %document style grand article journal
\usepackage{graphicx}%g�re les graphiques
\usepackage{here} %place les photos o� il faut
\usepackage{array} %faire des tableaux plus compliqu�s

\usepackage{vmargin}
\usepackage{amssymb}
\usepackage{mathrsfs}
\usepackage[all]{xy}
\usepackage[usenames,dvipsnames]{color}%64 couleurs
\usepackage{soul}
\RequirePackage[colorlinks,linkcolor=blue,citecolor=LimeGreen,urlcolor=red]{hyperref} %ref actives en couleur au lieu de carres
\usepackage{amsmath}
\newtheorem{theorem}{Theorem}[section]

\newtheorem{lemma}[theorem]{Lemma}

\newtheorem{prop}[theorem]{Proposition}
\newtheorem{remark}[theorem]{Remark}

\numberwithin{equation}{section}

\newcommand{\R}{\mathbb{R}}

\newcommand{\N}{\mathbb{N}}
\newcommand{\Z}{\mathbb{Z}}

\newcommand{\T}{\mathbb{T}}

\newcommand{\func}[3]{#1 : #2 \longrightarrow #3}

\newcommand{\abs}[1]{\left|#1\right|}
\newcommand{\eps}{\varepsilon}
\newcommand{\norm}[1]{\left\|#1\right\|}

\renewcommand{\leq}{\leqslant}
\renewcommand{\geq}{\geqslant}
\renewcommand{\bar}{\overline}

\newcommand{\pa}[1]{\left(#1\right)}
\newcommand{\cro}[1]{\left[#1\right]}
\newcommand{\br}[1]{\left\{#1\right\}}
\newcommand\restr[2]{{% we make the whole thing an ordinary symbol
  \left.\kern-\nulldelimiterspace % automatically resize the bar with \right
  #1 % the function
  \right|_{ #2} % this is the delimiter
  }}

\def\signmb{\bigskip \begin{center} {\sc
Marc Briant\par\vspace{3mm}
Laboratoire MAP5, Universit\'e Paris Descartes\par
45, rue des Saint-Pères,\par
75006, Paris, France\par
\vspace{3mm}
e-mail:} \tt{briant.maths@gmail.com} \end{center}}

\def\signjf{\bigskip \begin{center} {\sc
Julie Fournier\par\vspace{3mm}
Laboratoire MAP5, Universit\'e Paris Descartes\par
45, rue des Saint-Pères,\par
75006, Paris, France\par
\vspace{3mm}
e-mail:} \tt{julie.fournier@parisdescartes.fr} \end{center}}

\begin{document} 

\title[Isotropic diffeomorphisms for deformed random fields]{Isotropic diffeomorphisms: solutions to a differential system for a deformed random fields study}
\author{Marc Briant}
\author{Julie Fournier}
%\thanks{}
%\thanks{}

\begin{abstract}
This Note presents the resolution of a differential system on the plane that translates a geometrical problem about isotropic deformations of area and length. The system stems from a probability study on deformed random fields \cite{F17}, which are the composition of a random field with invariance properties defined on the plane with a deterministic diffeomorphism. The explicit resolution of the differential system allows to prove that a weak notion of isotropy of the deformed field, linked to its excursion sets, in fact coincides with the strong notion of isotropy. The present Note first introduces the probability framework that gave rise to the geometrical issue and then proposes its resolution.

\end{abstract}

\maketitle

\vspace*{10mm}
%\smallskip

%\textbf{Keywords:} 

%%\smallskip
%%\textbf{AMS Subject Classification}: 82C40 Kinetic theory of gases,
%%76P05 Rarefied gas flows, Boltzmann equation, 54C70 Entropy, 60J75
%%Jump processes.
%

\tableofcontents

\input{Isotropic_deformations}

%% APPENDIX %%%%

%\appendix
%
%\input{appendix}

%
% Pour une biblio classe
\bibliographystyle{acm}
\bibliography{bibliography_Isotropic_deformations}

%%%% Pour une biblio manuelle
%%\newpage
%%\include{bibliography}

% On met les signatures
\bigskip
\signmb
\signjf

\end{document}

%% file: Isotropic_deformations.tex
The motivation for the result featured in the present article originates from a probability problem about deformed random fields. Indeed, one of the main results of \cite{F17} needed a complete characterization of isotropic deformed fields and it turned out such a description was given by solutions to a system of nonlinear partial differential equations. The resolution of that system is a major step in the proof, yet it is completely independent. Moreover, its analytical flavour as well as the geometric classification it contains makes it interesting on its own and out of step with the probability nature of \cite{F17}.
\par Geometrically, the aim is to investigate the class of planar transformations $\func{F}{\R^2}{\R^2}$ that are $C^2$ and transform isotropically areas of rectangles and lengths of segments: 
\begin{equation}\label{eq:geometric}
\forall \varphi \in SO(2),\:\forall i \in \br{1,2}, \quad l_i(F\circ \varphi (E)) = l_i(F(E))
\end{equation}
where $l_i$ stands for the Lebesgue measure in $\R^1$ or $\R^2$ depending on $E$ being a segment (embedded in $\R^1$) or a rectangle (viewed as a surface embedded in $\R^2$). Such a property boils down to the fact that both the norms of each of the columns of the cartesian Jacobian matrix of the polar form  of $F$ and its determinant are radial: introducing $\T=\R/2\pi\Z$ the one-dimensional torus 
\begin{equation}\label{eq:jacobian}
\begin{split}
\forall (r,\theta) \in \R^+\times \T, \quad &\norm{\mbox{Jac}_F(r,\theta)\cro{\begin{array}{c} 1 \\ 0 \end{array}}} = g(r),\quad \norm{\mbox{Jac}_F(r,\theta)\cro{\begin{array}{c} 0 \\ 1 \end{array}} }= h(r)
\\&\mbox{det}\pa{\mbox{Jac}_F}(r,\theta) = f(r),
\end{split}
\end{equation}
where $f$, $g$ and $h$ are $C^2(\R^{+*},\R)$.
\par In this Note we prove that the class $(\ref{eq:geometric})$ is exactly the family of ``spiral deformations'' described in polar coordinates by:
\begin{equation}\label{eq:spiral}
\forall (r,\theta) \in \R^+\times \T,\quad R(r,\theta) = \mathcal{R}(r) \quad\mbox{and}\quad \Theta(r,\theta)=\pm\theta + \bar{\Theta}(r).
\end{equation}

The main goal of the present Note is therefore to introduce the background and one of the main results of \cite{F17} in order to motivate the analytic problem and then to solve it.

\section{Characterization of isotropy in deformed random \textit{via}  excursion sets}

All the random fields mentioned in this introduction are defined on $\R^2$, take real values and we furthermore assume that they are Gaussian.
\par A deformed random field is constructed with a regular, stationary and isotropic random field $X$ composed with a deterministic diffeomorphism $F$ such that $F(0)=0$. The result of this composition is a random field $X\circ F$.
Stationarity, respectively isotropy (refered to in the following as strong isotropy), consists in an invariance of the law of a random field under translations, respectively under rotations in $\R^2$. Even though the underlying field $X$ is isotropic, the deformed random fields constructed with $X$ are generally not. It is however possible to characterize explicitely a diffeomorphism $F$ such that for any underlying field $X$, the deformed field $X\circ F$ is strongly isotropic. Such diffeomorphisms are exactly the spiral diffeomorphisms introduced above $(\ref{eq:spiral})$.
\par The objective in \cite{F17} is to study a deformed field using sparse information, that is, the information provided by excursion sets of the field over some basic subsets in $\R^2$. If a real number $u$ is fixed, the excursion set of the field $X\circ F$ above level $u$ over a compact set $T$ is the random set
\[A_u(X\circ F,T):=\{t\in T\,/\, X(F(t))\geq u\}.
\]
One useful functional to study the topology of sets is the Euler characteristic, denoted by $\chi$. Heuristically, the Euler characteristic of a one-dimensional compact regular set is simply the number of intervals in this set; the Euler characteristic of a two-dimensional compact regular set is the number of connected components minus the number of holes in this set. 
\par A rotational invariance condition of the mean Euler characteristic of the excursion sets of $X\circ F$ over rectangles is then introduced as a weak isotropy property. More precisely, a random field $X\circ F$ is said to satisfy this weak isotropy property if for any real $u$, for any rectangle $T$ in $\R^2$ and for rotation
$\varphi$,
\begin{equation}  \label{invariancepropEuler}
\mathbb{E}[\chi(A_u(X\circ F, \varphi(T)))]=\mathbb{E}[\chi(A_u(X\circ F, T))].
\end{equation} 
The latter condition is in particular clearly true if the deformed field $X\circ F$ is strongly isotropic or, in other words, if $F$ is a spiral diffeomorphism. Provided that we add some assumptions on $X$, for any rectangle $T$ in $\R^2$, the expectation of $\chi(A_u(X\circ F,T))$ can be expressed as a linear combination of the area of the set $F(T)$ and the length of its frontier, with coefficients depending on $u$ only and not on the precise law of $X$.
\par Consequently, it occurs that Condition $(\ref{invariancepropEuler})$ is equivalent to Condition $(\ref{eq:geometric})$ and therefore to Condition $(\ref{eq:jacobian})$ on $F$ in the present Note. Theorem \ref{theo:expression functions} that we are going to demonstrate in this Note therefore implies that the spiral diffeomorphisms are the only solutions. This means that the associated deformed field is strongly isotropic, as explained before. As a result, the weak definition of isotropy coincides with the strong definition, as far as deformed fields are concerned.
\par From a practical point of view, a major consequence is that we only need information contained in the excursion sets of a deformed random field $X\circ F$ (more precisely, Condition $(\ref{invariancepropEuler})$ fulfilled) to decide the issue of isotropy.

%%%%%%%%%%%%%%%%%%%%%%%%%%%%%%%%%%%%%%%%%%%%%%%%%%%%%%%%%%%%%%%%%%%%%%%%%%%%%%%%%%%%%%%%%
%%%%%%%%%%%%%%%%%%%%%%%%%%%%%%%%%%%%%%%%%%%%%%%%%%%%%%%%%%%%%%%%%%%%%%%%%%%%%%%%%%%%%%%%%
%%%%%%%%%%%%%%%%%%%%%%%%%%%%%%%%%%%%%%%%%%%%%%%%%%%%%%%%%%%%%%%%%%%%%%%%%%%%%%%%%%%%%%%%%

\section{Planar deformations modifying lengths and areas isotropically}

\par We now turn to the study of the class of $C^2$ planar transformations satisfying $(\ref{eq:geometric})$. Using a polar representation for such $F$ we translate the rotational invariant property $(\ref{eq:jacobian})$ into the following system of non-linear partial differential equations.

\bigskip
\begin{theorem}\label{theo:expression functions}
Let two functions $\func{R}{\R^+ \times \T }{\R^+}$ and $\func{\Theta}{\R^+ \times \T }{\T }$ be continuous on $\R^{+} \times \T $ and $C^2$ in $\R^{+*} \times \T $ that satisfy: $R(\cdot,\cdot)$ is surjective and $R(0,\cdot)$ is a constant function. Let $f$, $g$ and $h$ be $C^1$ functions from $\R^{+*}$ to $\R$ such that $f$ does not vanish.
Then the following differential equalities hold
\begin{eqnarray}
\forall (r,\theta) \in \R^{+*}\times \T ,\quad f(r) &=& R\partial_r R \partial_\theta \Theta - R \partial_\theta R \partial_r \Theta  \label{eq:scalar product}
\\ g(r) &=& \pa{\partial_r R}^2 + \pa{R\partial_r\Theta}^2 \label{eq:norm partial r}
\\ h(r) &=& \pa{\partial_\theta R}^2 + \pa{R\partial_\theta\Theta}^2 \label{eq:norm partial theta}
\end{eqnarray}
if and only if there exist $\eps_1$ and $\eps_2$ in $\br{-1,1}$ and $\Theta_0$ in $\T$ such that
\begin{itemize}
\item[(i)] $h$ is strictly increasing and continuous on $\R^+$ with $h(0)=0$;
\item[(ii)] for all $r >0$, $f(r) = \eps_1 \frac{h'(r)}{2}$ and $g(r)h(r) \geq f^2(r)$;
\item[(iii)] the functions $R$ and $\Theta$ are given by
\end{itemize}
\begin{equation*}
\begin{split}
\forall (r,\theta) \in \R^+\times\T,\quad  &R(r,\theta) = \sqrt{h(r)}
\\& \Theta(r,\theta) = \eps_1 \theta + \Theta_0 + \eps_2\int_0^r\frac{\sqrt{h(r_*)g(r_*)-f^2(r_*)}}{h(r_*)}\:dr_*.
\end{split}
\end{equation*}
\end{theorem}
\bigskip

Of important note is the fact that the assumptions made on $R$ and $\Theta$ before the differential system are here to ensure that they indeed describe a polar representant of a planar deformation $F$. Also note that the solutions obtained above are indeed spiral deformations $(\ref{eq:spiral})$.
\par The rest of this section is devoted to the proof of the theorem above. We first find an equivalent version of $(\ref{eq:scalar product})-(\ref{eq:norm partial r})-(\ref{eq:norm partial theta})$ that is not quadratic. Second, we prove that this equivalent problem can be seen as a specific case of a hyperbolic system of equations solely constraint to $(\ref{eq:norm partial theta})$. Finally we show that these two constraints necessarily imply Theorem \ref{theo:expression functions}.

%%%%%%%%%%%%%%%%%%%%%%%%%%%%%%%%%%%%%%%%%%%%%%%%%%%%%%%%%%%%%%%%%%%%%

\subsection{A non quadratic equivalent}

Here we prove the following proposition that gives the shape of the derivatives of $R$ and $\Theta$.

\bigskip
\begin{prop}\label{prop:equivalent system}
Let $R$, $\Theta$, $f$, $g$ and $h$ be functions as described by Theorem \ref{theo:expression functions}. Then, they satisfy the system $(\ref{eq:scalar product})-(\ref{eq:norm partial r})-(\ref{eq:norm partial theta})$ if and only if there exist $p$ in $\N$ and a continuous function $\func{\Phi}{\R^+ \times \T }{\T }$ such that
\begin{equation}\label{eq:equivalent system}
\begin{split}
\forall (r,\theta) \in \R^+\times \T , \quad \partial_r R &= \sqrt{g(r)}\cos \pa{\Phi(r,\theta)}
\\R\partial_r\Theta &= \sqrt{g(r)}\sin \pa{\Phi(r,\theta)}
\\ \partial_\theta R &= (-1)^p\sqrt{\frac{gh - f^2}{g}}(r)\cos \pa{\Phi(r,\theta)} - \frac{f}{\sqrt{g}}(r)\sin \pa{\Phi(r,\theta)}
\\ R\partial_\theta \Theta &= (-1)^p\sqrt{\frac{gh - f^2}{g}}(r)\sin \pa{\Phi(r,\theta)} + \frac{f}{\sqrt{g}}(r)\cos \pa{\Phi(r,\theta)}.
\end{split}
\end{equation}
\end{prop}
\bigskip

\begin{proof}[Proof of Proposition \ref{prop:equivalent system}]
First, functions satisfying $(\ref{eq:equivalent system})$ are solutions to our original system $(\ref{eq:scalar product})-(\ref{eq:norm partial r})-(\ref{eq:norm partial theta})$ by mere computation.
\par Now assume that the functions are solutions of $(\ref{eq:scalar product})-(\ref{eq:norm partial r})-(\ref{eq:norm partial theta})$. The key is to see the quantities involved as complex numbers functions: $Z_1 = \partial_r R + i R\partial_r\Theta$ and $Z_2(r,\theta) = R\partial_\theta\Theta - i \partial_\theta R$. Then the system $(\ref{eq:scalar product})-(\ref{eq:norm partial r})-(\ref{eq:norm partial theta})$ translates into
\begin{equation*}
\begin{split}
\forall (r,\theta)\in \R^+\times \T , \quad &\abs{Z_1(r,\theta)}^2 = g(r)
\quad\mbox{and}\quad \abs{Z_2(r,\theta)}^2 = h(r) 
\\&\mbox{Re}\pa{ Z_1(r,\theta)\bar{Z_2(r,\theta)}} = f(r).
\end{split}
\end{equation*}
We first note that $f(r)^2 \leq g(r)h(r)$ and recalling that $f$ never vanishes on $\R^+$ it follows that neither $g$ nor $h$ can be null on $R^+$. Therefore since
$$g+h\pm 2f = \pa{\partial_r R \pm R\partial_\theta \Theta}^2 + \pa{\partial_\theta R \mp R\partial_r\Theta}^2 \geq 0$$
it follows $2\abs{f} \leq g(r)+h(r)$ and therefore we must in fact have
\begin{equation}\label{inequality fgh}
\forall r \in \R^{+}, \quad f(r)^2 < g(r)h(r).
\end{equation}
We can thus define the complex numbers
$$W_1(r,\theta) = \frac{Z_1(r,\theta)}{\sqrt{g(r)}},\quad W_2(r,\theta) = \sqrt{\frac{f(r)^2}{g(r)\pa{g(r)h(r)-f(r)^2}}}\cro{Z_1(r,\theta) - \frac{g(r)}{f(r)}Z_2(r,\theta)}.$$
which are of prime importance since they satisfy the following orthonormality property :
$$\forall (r,\theta)\in \R^+\times \T , \quad \abs{W_1(r,\theta)}^2  = \abs{W_2(r,\theta)}^2 = 1 \quad\mbox{and}\quad \mbox{Re}\pa{ W_1(r,\theta)\bar{W_2(r,\theta)}} = 0.$$
We deduce that there exist a continuous function $\func{\Phi}{\R^+ \times \T }{\T }$ and an integer $p \geq 0$ such that
$$\forall (r,\theta) \in \R^+\times \T , \quad W_1(r,\theta) = e^{i\Phi (r,\theta)} \quad\mbox{and}\quad W_2(r,\theta) = e^{-i\Phi (r,\theta) + (2p +1)\frac{\pi}{2}}.$$
Coming back to the original $Z_1, Z_2$ and then to $R$ and $\Theta$ concludes the proof.
\end{proof}

%%%%%%%%%%%%%%%%%%%%%%%%%%%%%%%%%%%%%%%%%%%%%%%%%%%%%%%%%%%%%%%%%%%%%

\subsection{A hyperbolic system under constraint}

We now find a more general system of equations satisfied by the functions we are looking for as well as a restrictive property that defines them.

\bigskip
\begin{lemma}\label{lem:hyperbolic system}
Let $R$, $\Theta$, $f$, $g$ and $h$ be functions as described by Theorem \ref{theo:expression functions}. Then, they satisfy the system $(\ref{eq:scalar product})-(\ref{eq:norm partial r})-(\ref{eq:norm partial theta})$ if and only if they satisfy $(\ref{eq:norm partial theta})$ and there exist $\func{\alpha, \beta}{\R^+}{\R}$ continuous with $\beta(r) >0$ such that
\begin{equation}\label{eq:hyperbolic}
\begin{split}
\forall (r,\theta) \in \R^{+*}\times \T , \quad \partial_\theta R &= \alpha(r) \partial_r R -\beta(r) R \partial_r \Theta
\\R \partial_\theta \Theta &= \alpha(r) R \partial_r \Theta +\beta(r) \partial_r R.
\end{split}
\end{equation}
\end{lemma}
\bigskip

\begin{remark}
Even if this set of equations still seems non-linear, it actually is linear in $X=(\mbox{ln}(R), \Theta)$. It indeed satisfies a vectorial transport equation $\partial_\theta X + A(r)\partial_r X =0$ with $A(r)$ being skew-symmetric and invertible. This equation is however non trivial as even in the case $\alpha(r) = 0$ we are left to solve $\partial_\theta \cro{\mbox{ln}(R)} = -\beta(r) \partial_\theta \Theta$ and $\partial_\theta \Theta = \beta(r) \partial_\theta \cro{\mbox{ln}(R)}$. And so $\mbox{ln}(R)$ and $\Theta$ are both solutions to $\partial_{\theta\theta} f + \beta(r)^2\partial_{rr}f =0$. For more on this subject we refer the reader to \cite{Serre99}.
\end{remark}
\bigskip

\begin{proof}[Proof of Lemma \ref{lem:hyperbolic system}]
The necessary condition follows directly from the set of equations $(\ref{eq:equivalent system})$ given in Proposition \ref{prop:equivalent system}, denoting $\alpha(r) =(-1)^p\frac{\sqrt{gh - f^2}}{g}(r)$ and $\beta(r) = f(r)/g(r)$ and dividing by $R >0$.
\par The sufficient condition follows by direct computation from $(\ref{eq:hyperbolic})$ and $(\ref{eq:norm partial theta})$, defining $g(r) = \frac{h(r)}{\alpha^2(r) + \beta^2(r)}$ and $f(r) = \beta(r)g(r) >0$.
\end{proof}
\bigskip

We now show that solutions to the hyperbolic system that are constraint by $(\ref{eq:norm partial theta})$ must satisfy that $R$ is radially symmetric.

\bigskip
\begin{prop}\label{prop:isotrope R}
Let $R,\:\Theta$ be solution to $(\ref{eq:hyperbolic})$ with $R$ and $\Theta$ verifying the assumptions of Theorem \ref{theo:expression functions}. Suppose that $(R,\Theta)$ also satisfies $(\ref{eq:norm partial theta})$; then $R$ is isotropic: for all $(r,\theta)$ in $\R^+\times \T $, $R(r,\theta) = \mathcal{R}(r)$ with moreover $\mathcal{R}(0)=0$ and  $\mathcal{R}'(r) >0$ for all $r>0$.
\end{prop}
\bigskip

\begin{proof}[Proof of Proposition \ref{prop:isotrope R}]
Let us first prove that $R(r,\theta)=0$ if and only if $r=0$.
\par The surjectivity of $R$ implies that there exists $(r_0,\theta_0)$ such that $R(r_0,\theta_0)=0$. If $r_0 \neq 0$ then by Lemma \ref{lem:hyperbolic system} $R$ and $\Theta$ satisfy $(\ref{eq:scalar product})$ at $(r_0,\theta_0)$ and thus $f(r_0) =0$. This contradicts the fact that $f$ does not vanish. Therefore $r_0=0$ and since $R(0,\cdot)$ is constant we get that
\begin{equation}\label{null R0}
\forall \theta \in\T, \quad R(0,\theta) =0.
\end{equation}
Recall that $R$ is positive and it follows
\begin{equation}\label{positive derivative}
\exists r_0 >0,\:\forall r \in (0,r_0],\:\forall \theta \in \T,\quad \partial_r R(r,\theta) \geq 0.
\end{equation}

\par We now turn to the study of local extrema of $R(r,\cdot)$. For a fixed $r>0$, if $\phi$ is a local extremum of $R(r,\cdot)$ then $\partial_\theta R(r,\phi) =0$ and also, thanks to $(\ref{eq:hyperbolic})$, $\alpha(r) \partial_r R (r,\phi) = \beta(r)R \partial_r \Theta (r,\phi)$. Plugging these equalities inside the equation satisfied by $R\partial_\theta \Theta$ in $(\ref{eq:hyperbolic})$ and since $\beta(r)\neq 0$ we get
$$\pa{R\partial_\theta \Theta}(r,\phi) = \frac{\alpha^2(r) + \beta^2(r)}{\beta(r)} \partial_r R (r,\phi).$$
Then we can apply the constraint on the angular derivatives $(\ref{eq:norm partial theta})$ to conclude
\begin{equation}\label{derivative extrema}
\forall (r,\phi) \:\mbox{such that}\:\partial_\theta R(r,\phi) =0,\quad  \pa{\partial_r R (r,\phi)}^2 = \frac{\beta^2(r)}{(\alpha^2(r)+\beta^2(r))^2} h(r).
\end{equation}

\bigskip
Finally, for any $r>0$, $R(r,\cdot)$ is $2\pi$-periodic and thus has a global maximum and a global minimum on $\T$. Let us choose $r_1$ in $(0,r_0)$ where $r_0$ has been defined in $(\ref{positive derivative})$. Functions $R$ and $\Theta$ are $C^2$ in $\R^{+*}\times \T$ and for any $\theta$, the determinant of the Jacobian of $(R,\theta)$ is equal to $f(r_1)/R(r_1)$, by $(\ref{eq:scalar product})$. This determinant does not vanish so by the local inverse function theorem there exist $[r_1-\delta_1,r_1+\delta_1]\subset [0,r_0]$ and two functions $\func{\phi_m,\:\phi_M}{[r_1-\delta_1 ,r_1+\delta_1]}{\T}$ that are $C^1$ and such that
$$\forall r>0, \quad R (r,\phi_m(r)) = \min\limits_{\theta \in [-\pi,\pi]} R(r,\theta) \quad\mbox{and}\quad R (r,\phi_M(r)) = \max\limits_{\theta \in [-\pi,\pi]} R(r,\theta).$$
By definition, for all $r$ in $[r_1-\delta_1,r_1+\delta_1]$, $\partial_\theta R (r,\phi_{m/M}(r)) = 0$ and because $\phi_m$ and $\phi_M$ are $C^1$ this implies:
\begin{equation}\label{partial r extrema}
\forall r \in[r_1-\delta_1,r_1+\delta_1], \quad \partial_r\pa{R (r,\phi_{m/M}(r))} = \partial_r R (r,\phi_{m/M}(r)).
\end{equation}
Thanks to $(\ref{partial r extrema})$, $(\ref{derivative extrema})$ and $(\ref{positive derivative})$ we thus obtain $\partial_r\pa{R (r,\phi_{m}(r))}= \partial_r\pa{R (r,\phi_{M}(r))}$ on $[r_1-\delta_1,r_1+\delta_1]$. This implies for all $r \in[r_1-\delta_1,r_1+\delta_1]$: 
$$\min\limits_{\theta \in [-\pi,\pi]} R(r,\theta) =\max\limits_{\theta \in [-\pi,\pi]} R(r,\theta) + \cro{\min\limits_{\theta \in [-\pi,\pi]} R(r_1-\delta_1,\theta)-\max\limits_{\theta \in [-\pi,\pi]} R(r_1-\delta_1,\theta)}.$$

\par To conclude we iterate: either $r_1-\delta_1  =0$ and we define $r_2=0$ or we can start our argument again with $r_2 = r_1-\delta_1$. Iterating the process we construct a sequence $(r_n)_{n\in\N^*}$ either strictly decreasing or reaching $0$ at a certain step and such that
\begin{equation}\label{iterating conclusion}
\begin{split}
\forall r \in [r_{n+1},r_1], \: \min\limits_{\theta \in [-\pi,\pi]} R(r,\theta) =&\max\limits_{\theta \in [-\pi,\pi]} R(r,\theta) + \min\limits_{\theta \in [-\pi,\pi]} R(r_{n+1},\theta)-\max\limits_{\theta \in [-\pi,\pi]} R(r_{n+1},\theta).
\end{split}
\end{equation}
This sequence thus converges to $r_\infty\geq 0$. If $r_\infty \neq 0$ then we could start our process again at $r_\infty$ and construct another decreasing sequence still satisfying $(\ref{iterating conclusion})$. In the end we will construct a sequence converging to $0$ so without loss of generality we assume that $r_\infty=0$.
\par Hence, since $R$ is continuous on $\R^+\times \T$ and $(\ref{null R0})$ holds true, it follows by taking the limit as $n$ tends to $\infty$ in $(\ref{iterating conclusion})$: 
$$\forall r \in [0,r_1], \quad \min\limits_{\theta \in [-\pi,\pi]} R(r,\theta) =\max\limits_{\theta \in [-\pi,\pi]} R(r,\theta).$$
The equality above implies that $\theta \mapsto R(r,\theta)$ is constant for any $r\leq r_1$.

\bigskip
The rotational invariance of $R(r,\cdot)$ holds for any $r_1 <r_0$ where $r_0$ is such that $(\ref{positive derivative})$ holds true. Therefore, denoting $r_M = \sup\br{r>0: \:\forall \theta \in \T,\quad  \partial_r R(r,\theta) \geq 0}$ it follows that for any $r\leq r_M$, $\theta \mapsto R(r,\theta)$ is constant. Since $R(\cdot,\cdot)$ is $C^2$ in $\R^{+*}\times\T$ we infer
$$\forall r \in (0,r_M),\:\forall \theta \in \T, \quad \partial_\theta R(r,\theta) =0 .$$
Suppose that $r_M <+\infty$ then by continuity of $\partial_\theta R$ we get $\partial_\theta R(r_M,\theta) =0$ for all $\theta$. But the definition of $r_M$ implies the existence of $\theta_M$ such that $\partial_r R(r_M,\theta_M) =0$. It follows that
$$\partial_\theta R(r_M,\theta_M) =0 \quad\mbox{and}\quad \partial_r R(r_M,\theta_M) =0.$$
Plugging the above inside the constraint $(\ref{eq:norm partial theta})$ yields $f(r_M) =0$ which is a contradiction since $r_M>0$.
\par We thus conclude that $r_M=+\infty$ and that $\theta \mapsto R(r,\theta)$ is invariant for any $r\geq 0$.

\end{proof}

%%%%%%%%%%%%%%%%%%%%%%%%%%%%%%%%%%%%%%%%%%%%%%%%%%%%%%%%%%%%%%%%%%%%%

\subsection{Isotropic solutions and proof of Theorem \ref{theo:expression functions}}

\begin{proof}[Proof of Theorem \ref{theo:expression functions}]
Let  us consider $R$, $\Theta$, $f$, $g$ and $h$ as in the statement of Theorem \ref{theo:expression functions} and satisfying the system $(\ref{eq:scalar product})-(\ref{eq:norm partial r})-(\ref{eq:norm partial theta})$. Thanks to Proposition \ref{prop:isotrope R}, there exists $\func{\mathcal{R}}{\R^+}{\R^+}$
continuous on $\R^+$ and $C^2$ on $\R^{+*}$ such that $\mathcal{R}(0)=0$, $\mathcal{R}'(r) >0$ and $R(r,\theta) = \mathcal{R}(r)$ for any $(r,\theta)$ in $\R^+\times \T$.
\par First, we recall $(\ref{inequality fgh})$: $h(r)g(r) > f(r)^2$ for $r>0$ and since $f$ does not vanish it follows that $h(r)>0$ and $g(r)>0$ for any $r>0$.
\par Then, thanks to $(\ref{eq:norm partial theta})$ we infer $\abs{\partial_\theta \Theta (r,\theta)} = \sqrt{h(r)}/\mathcal{R}(r)$. The right-hand side does not vanish so neither does $\partial_\theta \Theta$. By continuity it keeps a fixed sign $s$ in $\br{-1,1}$.
$$\forall (r,\theta) \in \R^{+*}\times \T,\quad \partial_\theta \Theta (r,\theta) = s\frac{\sqrt{h(r)}}{\mathcal{R}(r)}.$$
\par We now recall that $\Theta$ maps $\R^+\times\T$ to $\T$ and since the right-hand side does not depend on $\theta$ such an equality implies that
$$\exists \eps \in \br{-1,0,1},\:\forall r>0, \quad \frac{\sqrt{h(r)}}{\mathcal{R}(r)} = \eps.$$
Recalling that $\mathcal{R}(0)>0$ we deduce that $\eps =1$ and thus
\begin{equation}\label{final R final Theta}
\forall r\in \R^+, \quad \mathcal{R}(r) = \sqrt{h(r)} \quad\mbox{and}\quad \Theta(r,\theta) = \pm \theta + \bar{\Theta}(r)
\end{equation}
where $\bar{\Theta}$ is a function from $\R^+$ to $\T$.
\par At last, we use $(\ref{eq:scalar product})$ to obtain $f(r) = \pm  h'(r)/2$ for any $r>0$. The hyperbolic system $(\ref{eq:hyperbolic})$, with $\alpha(r) =(-1)^p\frac{\sqrt{gh - f^2}}{g}(r)$ ($p$ an integer defined in Proposition \ref{prop:equivalent system})) and $\beta(r) = f(r)/g(r)$, yields
$$\forall (r,\theta) \in \R^{+*}\times \T, \quad \partial_r\Theta(r,\theta) = \frac{\alpha(r) R(r,\theta) \partial_\theta \Theta (r,\theta) -\beta(r) \partial_\theta R(r,\theta)}{(\alpha^2(r) + \beta^2(r))R(r,\theta)}$$
which implies
\begin{equation}\label{final Thetabar}
\forall r \geq 0, \quad \bar{\Theta}(r) = \pm (-1)^p \int_0^r \frac{\sqrt{g(r_*)h(r_*) - f^2(r_*)}}{h(r_*)}\:dr_* + \bar{\Theta}(0).
\end{equation}
This concludes the proof because equations $(\ref{inequality fgh})$, $(\ref{final R final Theta})$ and $(\ref{final Thetabar})$ are exactly the conditions stated in Theorem \ref{theo:expression functions}. The sufficient condition is checked by direct computations.
\end{proof}